\documentclass[12pt,etds]{amsart}
\usepackage{amsmath}
\usepackage{amssymb}
\usepackage{amsfonts}
\usepackage{amsthm}
\usepackage{enumerate}
\usepackage{tikz}
\usepackage{url}
\usetikzlibrary{matrix}

\usepackage{pb-diagram}
\usepackage[all,cmtip]{xy}

\newtheorem{theorem}{Theorem}[section]
\newtheorem{lemma}[theorem]{Lemma}
\newtheorem{prop}[theorem]{Proposition}
\newtheorem{cor}[theorem]{Corollary}
\theoremstyle{remark}

\newtheorem{remark}[theorem]{Remark}

\def\N{{\mathbb N}}

\def\C{{\mathbb C}}

\def\TT{{\mathbb T}}

\def\Z{{\mathbb Z}}

\def\A{{\mathcal{A}}}

\def\E{{\mathcal{E}}}
\def\U{{\mathcal{U}}}
\def\K{{\mathcal{K}}}
\def\H{{\mathcal{H}}}
\def\T{{\mathcal{T}}}
\def\I{{\mathcal{I}}}
\def\J{{\mathcal{J}}}
\def\L{{\mathcal{L}}}

\def\M{{\mathcal{M}}}
\def\c{{\bf c}}
\newcommand{\clsp}{\overline{\operatorname{span}}}

\newcommand{\id}{\operatorname{id}}

\newcommand{\piso}{\operatorname{piso}}
\newcommand{\iso}{\operatorname{iso}}

\newcommand{\Prim}{\operatorname{Prim}}

\newcommand{\clu}{\widehat{u}}
\newcommand{\whitesquare}{\hfill $\whitesquare$\newline\vspace{0.4cm}}

\makeatletter
\@namedef{subjclassname@2020}{\textup{2020} Mathematics Subject Classification}
\makeatother

\numberwithin{equation}{section}

\begin{document}

\title[The ideal structure of the semigroup crossed product]
{The primitive ideal space of the partial-isometric crossed product corresponding to an action induced by the multiplication on $p$-adic integers}

\author[Saeid Zahmatkesh]{Saeid Zahmatkesh}
\address{Mathematics and Statistics with Applications (MaSA), Department of Mathematics, Faculty of Science, King Mongkut's University of Technology Thonburi, Bangkok 10140, THAILAND}
\email{saeid.zk09@gmail.com, saeid.kom@kmutt.ac.th}




\begin{abstract}
For an odd prime $p$, we consider an action $\alpha$ of the semigroup $\N^{2}$ on the algebra $C(\Z_p)$ induced by the multiplication on
the compact topological ring of $p$-adic integers $\Z_p$. We then identify all primitive ideals of the partial-isometric crossed product
$C(\Z_{p})\rtimes_{\alpha}^{\piso} \N^{2}$ and describe the hull-kernel closures of the subsets of its primitive ideal space.
\end{abstract}
\maketitle

\section{Introduction}
\label{sec:intro}
Let $\N^{2}$ be the positive cone of the abelian lattice-ordered group $\Z^{2}$. Let $p$ be an odd prime and $\Z_{p}$ the compact topological ring of $p$-adic integers.
We pick another odd prime $q$ and consider the action of $\N^{2}$ on $\Z$ defined by
$$(m,n)\cdot x:=p^{m}q^{n}x.$$
This action indeed extends to an action of $\N^{2}$ on $\Z_p$, and therefore, we obtain a semigroup dynamical system
$(C(\Z_p),\N^{2},\alpha)$, where
\begin{itemize}
\item[$\bullet$] $C(\Z_p)$ is the $C^*$-algebra of all continuous complex-valued functions on $\Z_p$; and
\item[$\bullet$] $\alpha$ is an action of $\N^{2}$ on $C(\Z_p)$ by endomorphisms such that
\[
\alpha_{(m,n)}(f)(x)=
   \begin{cases}
      f(p^{-m}q^{-n}x) &\textrm{if}\empty\ \text{$x\in p^{m}q^{n}\mathbb{Z}_p$},\\
      0 &\textrm{otherwise}
   \end{cases}
\]
\end{itemize}
for all $(m,n)\in \N^{2}$, $f\in C(\Z_p)$, and $x\in \Z_p$. Note that the action $\alpha$ is not unital, that is, $\alpha_{(m,n)}(1)\neq 1$.
Now, our goal in the present work is to study the primitive ideal space of the partial-isometric crossed product $C(\Z_{p})\rtimes_{\alpha}^{\piso} \N^{2}$ of the system $(C(\Z_p),\N^{2},\alpha)$.
To do so, we first consider the (extendibly) invariant ideal $C_{0}(\Z_{p}\setminus\{0\})$ of $C(\Z_{p})$ which gives rise to ideal $$C_{0}(\Z_{p}\setminus\{0\})\rtimes_{\alpha}^{\piso} \N^{2}$$ and the
quotient algebra $$C(\Z_{p})\rtimes_{\alpha}^{\piso} \N^{2}/C_{0}(\Z_{p}\setminus\{0\})\rtimes_{\alpha}^{\piso} \N^{2}\simeq \T(\Z^{2}).$$ Then, all primitive ideals are determined, derived from
$\Prim(C_{0}(\Z_{p}\setminus\{0\})\rtimes_{\alpha}^{\piso} \N^{2})$ and $\Prim (\T(\Z^{2}))$. Finally, by identifying spanning elements for the ideals coming from
$\Prim(C_{0}(\Z_{p}\setminus\{0\})\rtimes_{\alpha}^{\piso} \N^{2})$, we compute the hull-kernel closures of the subsets of $\Prim(C(\Z_{p})\rtimes_{\alpha}^{\piso} \N^{2})$.

The dynamical system $(C(\Z_p),\N^{2},\alpha)$ along with the ideal $C_{0}(\Z_{p}\setminus\{0\})$ appears, first, in \cite{LPR}, where an isometric crossed product
$C(\Z_{p}\times \Z_q)\rtimes^{\iso} \N^{2}$ of the algebra $C(\Z_p\times \Z_q)\simeq C(\Z_{p})\otimes C(\Z_{q})$, which is
similar to the one that models the Hecke $C^{*}$-algebra of Bost and Connes, is studied. Following \cite{LPR}, it then appears in \cite[\S7]{SZ3} which discusses the partial-isometric crossed
$C(\Z_{p}\times \Z_q)\rtimes^{\piso} \N^{2}$, for which, a composition series of ideals is obtained (see \cite[Theorem 7.1]{SZ3}). Since the algebras
$C_{0}(\Z_{p}\setminus\{0\})\rtimes_{\alpha}^{\piso} \N^{2}$ and $\T(\Z^{2})$
both play role in the subquotients of that composition series of ideals, we strongly believe that the present work provides very useful information for
further investigations on the ideal structure of $C(\Z_{p}\times \Z_q)\rtimes^{\piso} \N^{2}$, where the study of its primitive ideal space might be considered in future works.

We begin with a preliminary section which recalls the theory of the partial-isometric crossed products in brief. In section \ref{sec:set-prim}, we determine all primitive ideals of
$C(\Z_{p})\rtimes_{\alpha}^{\piso} \N^{2}$ coming from the ideal $C_{0}(\Z_{p}\setminus\{0\})\rtimes_{\alpha}^{\piso} \N^{2}$ and the corresponding quotient algebra $\T(\Z^{2})$.
In section \ref{sec:closures}, the final section, we identify the spanning elements of the ideals coming from $\Prim(C_{0}(\Z_{p}\setminus\{0\})\rtimes_{\alpha}^{\piso} \N^{2})$ and describe
the hull-kernel closures of the subsets of $\Prim(C(\Z_{p})\rtimes_{\alpha}^{\piso} \N^{2})$.

\section{Preliminaries}
\label{sec:pre}

Let $S$ be the positive cone of an abelian lattice-ordered group $G$, and $A$ a $C^{*}$-algebra (not necessarily unital). Suppose that
$\alpha$ is an action of the semigroup $S$ on $A$ by (extendible) endomorphisms. The trio $(A,S,\alpha)$ is called a
\emph{(semigroup) dynamical system}.
A \emph{Nica-covariant partial-isometric representation} of the dynamical system $(A,S,\alpha)$ is a pair $(\pi,V)$ consisting of
a non-degenerate representation $\pi$ of $A$ on a Hilbert space $H$ and a representation $V$ of $S$ on $H$ by partial isometries,
such that
$$\pi(\alpha_{s}(a))=V_{s}\pi(a)V_{s}^{*}\ \ \textrm{and}\ \ V_{s}^{*}V_{s}\pi(a)=\pi(a)V_{s}^{*}V_{s},$$
and
$$V_{r}^{*}V_{r}V_{s}^{*}V_{s}=V_{r\vee s}^{*}V_{r\vee s}\ \ \ (\textrm{\emph {Nica-covariance}})$$
for all $a\in A$ and $r,s\in S$, where $r\vee s$ denotes the supremum of $r$ and $s$. There is a $C^{*}$-algebra, denoted by
$A\rtimes_{\alpha}^{\piso} S$, corresponding to the system $(A,S,\alpha)$ in which the action $\alpha$ is decoded by partial-isometries. It is universal for the Nica-covariant partial-isometric representations of $(A,S,\alpha)$, that is, there is a bijection between the Nica-covariant partial-isometric representations of $(A,S,\alpha)$ and the non-degenerate representations of $A\rtimes_{\alpha}^{\piso} S$.
The non-degenerate representation of $A\rtimes_{\alpha}^{\piso} S$ corresponding to a covariant pair $(\pi,V)$ is denoted by
$\pi\rtimes^{\piso} V$, or simply $\pi\rtimes V$. The algebra $A\rtimes_{\alpha}^{\piso} S$ is called the \emph{partial-isometric crossed product}
of the system $(A,S,\alpha)$, which is generated by a universal covariant pair $(i_{A},i_{S})$ of $(A,S,\alpha)$ into the multiplier algebra $\M(A\rtimes_{\alpha}^{\piso} S)$. In fact, it is spanned by the elements
$$\{i_{S}(r)^{*}i_{A}(a)i_{S}(s): a\in A\ \textrm{and}\ r,s\in S\}.$$
Note that we have
$$\pi\rtimes V(i_{A}(a))=\pi(a)\ \ \textrm{and}\ \ \overline{\pi\rtimes V}(i_{S}(s))=V_{s}.$$
For more on the theory of the partial-isometric crossed products, readers are referred to \cite{Adji-Abbas, AZ,AZ2, Fowler, LZ, LR, SZ, SZ2,SZ3, SZ4, SZ5}. In particular, see in
\cite{SZ3} that the partial-isometric crossed products are actually defined for a larger class of semigroups, and by \cite{SZ2}, when $S$ is the positive cone of an abelian lattice-ordered group $G$,
$A\rtimes_{\alpha}^{\piso} S$ sits in a classical crossed product by the group $G$ as a full corner.

To see about the isometric crossed product $A\rtimes_{\alpha}^{\iso} S$ of $(A,S,\alpha)$ (another kind of semigroup crossed product corresponding to the system) that appears in
the present work, readers may refer to works such as \cite{ALNR, Fowler, Larsen, Stacey}.

\section{The primitive ideal space of $C(\Z_{p})\rtimes_{\alpha}^{\piso} \N^{2}$ as a set}
\label{sec:set-prim}

Consider the partial-isometric crossed product $(C(\Z_{p})\rtimes_{\alpha}^{\piso} \N^{2}, i_{C(\Z_{p})}, i_{\N^{2}})$ of the dynamical
system $(C(\Z_p),\N^{2},\alpha)$ (see \S1). Consider the extendibly $\alpha$-invariant ideal $C_{0}(\Z_{p}\setminus\{0\})$ of
$C(\Z_p)$ as in \cite{LPR, SZ3}, corresponding to which, we have the following
short exact sequence
\begin{align}
\label{ext-seq}
\xymatrix{
0 \ar[r] & C_{0}(\mathbb{Z}_{p}\setminus\{0\})\rtimes_{\alpha}^{\textrm{piso}} \mathbb{N}^{2} \ar[r] &
C(\mathbb{Z}_{p})\rtimes_{\alpha}^{\textrm{piso}} \mathbb{N}^{2} \ar[r]^-{\varphi} & \mathcal{T}(\mathbb{Z}^{2}) \ar[r] & 0
}
\end{align}
of $C^*$-algebras by \cite[Theorem 4.8]{SZ3}. Note that it is known that
$$[C(\Z_{p})/C_{0}(\Z_{p}\setminus\{0\})]\rtimes_{\tilde{\alpha}}^{\piso} \N^{2} \simeq \C\rtimes_{\id}^{\piso} \N^{2}$$
is isomorphic to Toeplitz algebra $\T(\Z^{2})$ (for example, see\cite{SZ2}). Thus, the primitive ideals of the algebra $C(\mathbb{Z}_{p})\rtimes_{\alpha}^{\textrm{piso}} \mathbb{N}^{2}$ are derived from the disjoint sets
$$\Prim(C_{0}(\mathbb{Z}_{p}\setminus\{0\})\rtimes_{\alpha}^{\textrm{piso}} \mathbb{N}^{2})\ \ \textrm{and}
\ \ \Prim(\mathcal{T}(\mathbb{Z}^{2})).$$
To identify the ideals derived from the open set $\Prim(C_{0}(\mathbb{Z}_{p}\setminus\{0\})\rtimes_{\alpha}^{\textrm{piso}} \mathbb{N}^{2})$, let $\U_{p}$ denote the compact group of the multiplicatively invertible elements of $\Z_{p}$, and $\c_{0}$ and $\c$ the algebras $C_{0}(\N)$ and $C(\N\cup\{\infty\})$, respectively.
By \cite[Lemma 2.4]{LPR}, there is the following isomorphism
$$C_{0}(\mathbb{Z}_{p}\setminus\{0\})\simeq C_{0}(\N\times \U_{p})\simeq \c_{0}\otimes C(\U_{p})$$
which splits the action $\alpha$ as a tensor product action $\tau\otimes \beta$, where $\tau$ is the action of $\N$ on $\c_{0}$ given by forward
shifts, and $\beta$ the action of $\N$ on the algebra $C(\U_{p})$ by automorphisms such that
$\beta_{n}(g)(x)=g(q^{-n}x)$ for all $g\in C(\U_{p})$ and $x\in \U_{p}$. Therefore, as it was discussed in \cite{SZ3}, the crossed product $$C_{0}(\mathbb{Z}_{p}\setminus\{0\})\rtimes_{\alpha}^{\textrm{piso}} \mathbb{N}^{2}\simeq
(\c_{0}\otimes C(\U_{p}))\rtimes_{\tau\otimes \beta}^{\piso} (\N\times \N)$$ is decomposed as the maximal tensor product
\begin{align}
\label{maxtensor}
(\c_{0}\rtimes_{\tau}^{\textrm{piso}} \N)\otimes_{\max} (C(\U_{p})\rtimes_{\beta}^{\piso} \N).
\end{align}
The crossed product $\c_{0}\rtimes_{\tau}^{\textrm{piso}} \N$ is a full corner in the algebra $\K(\ell^{2}(\N))\otimes \c$ of
compact operators (see \cite[Example 4.3]{AZ}). Let us denote it by $\A$ for convenience. Note that since the algebra
$\K(\ell^{2}(\N))\otimes \c$ is CCR (liminal), so is $\A$. It follows that $\A$ is nuclear, and therefore, the maximal tensor product in (\ref{maxtensor}) is just the minimal (spatial) one. Thus, we simply write
\begin{align}
\label{tensor}
C_{0}(\mathbb{Z}_{p}\setminus\{0\})\rtimes_{\alpha}^{\textrm{piso}} \mathbb{N}^{2}\simeq \A\otimes (C(\U_{p})\rtimes_{\beta}^{\piso} \N).
\end{align}
Now, it follows by \cite[Theorem B.45]{RW} that
\begin{align}
\label{prim1}
\Prim(C_{0}(\mathbb{Z}_{p}\setminus\{0\})\rtimes_{\alpha}^{\textrm{piso}} \mathbb{N}^{2})
\simeq\Prim \A\times \Prim(C(\U_{p})\rtimes_{\beta}^{\piso} \N),
\end{align}
where $$\Prim \A\simeq \Prim(\K(\ell^{2}(\N))\otimes \c)\simeq \Prim \c\simeq  \N\cup\{\infty\}.$$
For $\Prim(C(\U_{p})\rtimes_{\beta}^{\piso} \N)$, let $\Gamma_{q}$ denote the closed subgroup of $\U_{p}$ generated by
$q$. Since the action $\beta$ is given by automorphisms and $\Z$ acts freely on $\U_{p}$, \cite[Theorem 3.11]{LZ}
identifies $\Prim(C(\U_{p})\rtimes_{\beta}^{\piso} \N)$ with the disjoint union
$$\U_{p}\sqcup \Prim(C(\U_{p})\rtimes_{\beta} \Z),$$
and determines the open sets, precisely. However, $\Prim(C(\U_{p})\rtimes_{\beta} \Z)$ is indeed homeomorphic to the quotient (group)
$\U_{p}/\Gamma_{q}$, which is a finite set equipped with the discrete topology (see \cite{LPR}). Consequently,
\begin{align}
\label{prim2}
\Prim(C_{0}(\mathbb{Z}_{p}\setminus\{0\})\rtimes_{\alpha}^{\textrm{piso}} \mathbb{N}^{2})
\simeq (\N\cup\{\infty\}) \times (\U_{p}\sqcup \U_{p}/\Gamma_{q}).
\end{align}

For the ideals lifted from the closed set $\Prim \T(\Z^{2})$, since $\T(\Z^{2})\simeq \T(\Z)\otimes \T(\Z)$, again by
\cite[Theorem B.45]{RW}, we have
$$\Prim \T(\Z^{2})=\Prim \T(\Z) \times \Prim \T(\Z),$$
where $\Prim \T(\Z)$ is known. We use \cite[Example 3.9]{LZ}, in which, $\Prim \T(\Z)$ is described as the disjoint union $\{0\}\sqcup \TT$.
Therefore, one may determines the open sets in $\Prim \T(\Z^{2})$ via the product topology. However, since
$\T(\Z^{2})\simeq \C\rtimes_{\id}^{\piso} \mathbb{N}^{2}$, we apply \cite[Theorem 4.2]{SZ5}. It follows that
\begin{eqnarray*}
\begin{array}{rcl}
\Prim \T(\Z^{2})&=&\big(\{0\}\sqcup \TT\big)\times \big(\{0\}\sqcup \TT\big)\\
&=&\{(0,0)\}\sqcup \big(\TT\times \{0\}\big) \sqcup \big(\{0\}\times \TT\big) \sqcup \TT^{2}.
\end{array}
\end{eqnarray*}
Let us denote $\{(0,0)\}$ simply by $\{0\}$. Note that $\TT\times \{0\}$ and $\{0\}\times \TT$ are both homeomorphic to $\TT$.
However, the primitive ideals that they induce (parameterize) are different (see Lemma \ref{lem2}). So, to distinguish these two sets (and
the corresponding primitive ideals), they are denoted by $\dot{\TT}$ and $\ddot{\TT}$, respectively, for convenience. Consequently,
\begin{align}
\label{prim3}
\Prim \T(\Z^{2})=\{0\} \sqcup \dot{\TT} \sqcup \ddot{\TT} \sqcup \TT^{2},
\end{align}
in which, the open sets are precisely determined by \cite[Theorem 4.2]{SZ5}.

Thus, it follows by (\ref{prim2}) and (\ref{prim3}) that $\Prim(C(\Z_{p})\rtimes_{\alpha}^{\piso} \N^{2})$, as
a set, can be identified with the following disjoint union of eight sets:
\begin{eqnarray*}
\begin{array}{l}
(\N\cup\{\infty\}) \times (\U_{p}\sqcup \U_{p}/\Gamma_{q})\sqcup\{0\}\sqcup \dot{\TT}\sqcup\ddot{\TT} \sqcup \TT^{2}\\
=(\N\times \U_{p}) \sqcup (\N\times (\U_{p}/\Gamma_{q})) \sqcup \U_{p}\sqcup (\U_{p}/\Gamma_{q})
\sqcup\{0\}\sqcup \dot{\TT}\sqcup\ddot{\TT} \sqcup \TT^{2},
\end{array}
\end{eqnarray*}
where $\U_{p}$ and $(\U_{p}/\Gamma_{q})$ simply correspond to the sets $\{\infty\}\times \U_{p}$ and $\{\infty\}\times(\U_{p}/\Gamma_{q})$,
respectively. We are now ready to identify all primitive ideals starting with the ones lifted from the open set $\Prim(C(\Z_{p}\setminus\{0\})\rtimes_{\alpha}^{\piso} \N^{2})$.

For the following lemma, let $\{e_{r}:r\in \N\}$ denote the usual orthonormal basis of the Hilbert space $\ell^{2}(\N)$.
\begin{lemma}
\label{lem1}
Let $n\in\N$, $u\in\U_{p}$, and $\H_{n}$ be the Hilbert space $\C^{n+1}=\oplus_{i=0}^{n}\C$. Suppose that $T:\N\rightarrow B(\ell^{2}(\N))$ is
the representation of $\N$ on $\ell^{2}(\N)$ by isometries such that $T_{s}(e_{r})=e_{r+s}$, $P_{n}$ the projection $1-T_{n+1}T_{n+1}^{*}$ of $\ell^{2}(\N)$ on $\H_{n}$ and $\lambda:\Z\rightarrow B(\ell^{2}(\Z))$ the left regular representation. Let $\widehat{u}$ denote the equivalence class
$u\Gamma_{q}$ of $u$ in $\U_{p}/\Gamma_{q}$.
Define the following pairs of maps corresponding to $(n,u)$, $(n,\clu)$, $u$, and $\clu$, respectively:
\begin{itemize}
\item[(i)] $\rho:C(\Z_{p})\rightarrow B(\H_{n}\otimes \ell^{2}(\N))$ and $V\otimes T^{*}:\N^{2}\rightarrow B(\H_{n}\otimes \ell^{2}(\N))$
by
$$(\rho(f)\xi)(r,s)=f(p^{n-r}q^{-s}u)\xi(r,s)\ \textrm{for every}\ 0\leq r\leq n\ \textrm{and}\ s\in\N,$$
and
$$(r,s)\mapsto V_{r}\otimes T_{s}^{*}=(P_{n}T_{r}^{*}P_{n})\otimes T_{s}^{*}\ \textrm{for all}\ r,s\in \N,$$
respectively;

\item[(ii)] $\rho:C(\Z_{p})\rightarrow B(\H_{n}\otimes \ell^{2}(\Z))$ and $V\otimes \lambda:\N^{2}\rightarrow B(\H_{n}\otimes \ell^{2}(\Z))$
by
$$(\rho(f)\xi)(r,s)=f(p^{n-r}q^{s}u)\xi(r,s)\ \textrm{for every}\ 0\leq r\leq n\ \textrm{and}\ s\in\Z,$$
and
$$(r,s)\mapsto V_{r}\otimes \lambda_{s}=(P_{n}T_{r}^{*}P_{n})\otimes \lambda_{s}\ \textrm{for all}\ r,s\in \N,$$
respectively;

\item[(iii)] $\rho:C(\Z_{p})\rightarrow B(\ell^{2}(\N)\otimes \ell^{2}(\N))$ and
$T\otimes T^{*}:\N^{2}\rightarrow B(\ell^{2}(\N)\otimes \ell^{2}(\N))$
by
$$(\rho(f)\xi)(r,s)=f(p^{r}q^{-s}u)\xi(r,s)$$
and
$$(r,s)\mapsto T_{r}\otimes T_{s}^{*}\ \textrm{for all}\ r,s\in \N,$$
respectively; and

\item[(iv)] $\rho:C(\Z_{p})\rightarrow B(\ell^{2}(\N)\otimes \ell^{2}(\Z))$
and $T\otimes \lambda:\N^{2}\rightarrow B(\ell^{2}(\N)\otimes \ell^{2}(\Z))$
by
$$(\rho(f)\xi)(r,s)=f(p^{r}q^{s}u)\xi(r,s)\ \textrm{for all}\ r\in \N\ \textrm{and}\ s\in\Z,$$
and
$$(r,s)\mapsto T_{r}\otimes \lambda_{s}\ \textrm{for all}\ r,s\in \N,$$
respectively.
\end{itemize}
Each pair in above is a Nica-covariant partial-isometric representation of the system $(C(\Z_{p}),\N^{2},\alpha)$.
\end{lemma}

\begin{proof}
Calculations on the spanning elements show that each pair is indeed a Nica-covariant partial-isometric representation of
$(C(\Z_{p}),\N^{2},\alpha)$. However, we skip them here.
\end{proof}
Let
$$\Pi_{(n,u)}, \Pi_{(n,\widehat{u})}, \Pi_{u}\ \textrm{and}\ \Pi_{\clu}$$
denote the (unital) representations of $C(\Z_{p})\rtimes_{\alpha}^{\piso} \N^{2}$ corresponding to the covariant pairs (i)-(iv)
in Lemma \ref{lem1}, respectively. These representations are actually irreducible. To show this, it is enough to see that their restrictions
to the ideal $C_{0}(\Z_{p}\setminus\{0\})\rtimes_{\alpha}^{\piso} \N^{2}\simeq \A\otimes C(\U_{p})\rtimes_{\beta}^{\piso} \N$ are nonzero
and irreducible. In fact, we will see that these restrictions split as tensor product representations $\pi_{1}\otimes \pi_{2}$, where
$\pi_{1}$ and $\pi_{2}$ are known (nonzero) irreducible representations of $\A$ and $C(\U_{p})\rtimes_{\beta}^{\piso} \N$, respectively.
Therefore, $\pi_{1}\otimes \pi_{2}$ is irreducible (see\cite[Theorem B.45]{RW}).
\begin{prop}
\label{pi1&pi2}
Let $n\in\N$ and $u\in\U_{p}$, and $V:\N\rightarrow B(\H_{n})$, $T:\N\rightarrow B(\ell^{2}(\N))$, $T^{*}:\N\rightarrow B(\ell^{2}(\N))$, and $\lambda:\N\rightarrow B(\ell^{2}(\Z))$ be the representations of $\N$ by partial-isometries as in Lemma \ref{lem1}. Consider the following maps:
\begin{itemize}
\item[(a)] $\phi_{n}:\c_{0}\rightarrow B(\H_{n})$ defined by
$$(\phi_{n}(g)\xi)(r)=g(n-r)\xi(r)\ \textrm{for every}\ 0\leq r\leq n,\ g\in\c_{0}\ \textrm{and}\ \xi\in\H_{n};$$
\item[(b)] $M:\c_{0}\rightarrow B(\ell^{2}(\N))$ defined by
$$(M(g)\xi)(r)=g(r)\xi(r)\ \textrm{for every}\ r\in\N,\ g\in\c_{0}\ \textrm{and}\ \xi\in\ell^{2}(\N);$$
\item[(c)] $\pi_{u}:C(\U_{p})\rightarrow B(\ell^{2}(\N))$ defined by
$$(\pi_{u}(h)\xi)(s)=h(q^{-s}u)\xi(s)\ \textrm{for every}\ s\in\N,\ h\in C(\U_{p})\ \textrm{and}\ \xi\in\ell^{2}(\N);$$ and
\item[(d)] $\pi_{\clu}:C(\U_{p})\rightarrow B(\ell^{2}(\Z))$ defined by
$$(\pi_{\clu}(h)\xi)(s)=h(q^{s}u)\xi(s)\ \textrm{for every}\ s\in\Z,\ h\in C(\U_{p})\ \textrm{and}\ \xi\in\ell^{2}(\Z).$$
\end{itemize}
Then,
\begin{itemize}
\item[(i)] the pairs $(\phi_{n},V)$ and $(M,T)$ are partial-isometric covariant representations of the system $(\c_{0},\N,\tau)$ such that
the corresponding (nondegenerate) representations $\phi_{n}\rtimes V$ and $M\rtimes^{\piso} T$ of $\A=\c_{0}\rtimes_{\tau}^{\piso} \N$
are irreducible;
\item[(ii)] the pairs $(\pi_{u},T^{*})$ and $(\pi_{\clu},\lambda)$ are partial-isometric covariant representations of the system
$(C(\U_{p}),\N,\beta)$ such that corresponding (unital) representations $\pi_{u}\rtimes T^{*}$ and $\pi_{\clu}\rtimes^{\piso} \lambda$ of $C(\U_{p})\rtimes_{\beta}^{\piso} \N$ are irreducible
\end{itemize}

\end{prop}

\begin{proof}
For (i), we skip the calculations which show that the pair $(\phi_{n},V)$ is covariant. For the corresponding representation
$\phi_{n}\rtimes V$, note that the algebra $\c_{0}\rtimes_{\tau}^{\piso} \N$ sits in the crossed product $\c\rtimes_{\tau}^{\piso} \N$ as
an (essential) ideal (denoted by $\ker \varphi_{T}$) (see \cite[Corollary 3.1]{Adji-Abbas} and \cite[Proposition 5.1]{AZ2}).
Now, the restriction of the irreducible representation $\pi_{n}^{*}$ of $\c\rtimes_{\tau}^{\piso} \N$ given in \cite{AZ2}
to the ideal $\ker \varphi_{T}\simeq \c_{0}\rtimes_{\tau}^{\piso} \N$ is precisely the representation $\phi_{n}\rtimes V$. Thus,
it is irreducible. However, it is not difficult to see this directly. For the pair $(M,T)$, it is well-known that it is an
isometric covariant representation of $(\c_{0},\N,\tau)$ such that corresponding (nondegenerate) representation $M\rtimes^{\iso} T$ is an
isomorphism of the isometric crossed product $\c_{0}\rtimes_{\tau}^{\iso} \N$ onto the algebra of $\K(\ell^{2}(\N))$ of compact operators.
Now, if $\phi$ the natural surjection of $\c_{0}\rtimes_{\tau}^{\piso} \N$ onto $\c_{0}\rtimes_{\tau}^{\iso} \N$, we have
$M\rtimes^{\piso} T=(M\rtimes^{\iso} T)\circ\phi$. This implies that the image of $M\rtimes^{\piso} T$ is $\K(\ell^{2}(\N))$,
and therefore, $M\rtimes^{\piso} T$ is irreducible, such that $\ker(M\rtimes^{\piso} T)=\ker \phi$ is a full corner in the algebra
$\K(\ell^{2}(\N))\otimes \c_{0}$ (see \cite[Example 4.3]{AZ}).

For (ii), since the action $\beta$ in the system $(C(\U_{p}),\N,\beta)$ is given by automorphisms, we can apply the results of \cite{LZ}.
For every $u\in \U_{p}$, let $\varepsilon_{u}:C(\U_{p})\rightarrow \C$ be the evaluation map. By \cite[Proposition 3.1]{LZ}, it gives rise
to a partial-isometric covariant representation $(\pi,W)$ of $(C(\U_{p}),\N,\beta)$ on the Hilbert space $\ell^{2}(\N)$ such that
corresponding (unital) representation $\pi\rtimes W$ of $C(\U_{p})\rtimes_{\beta}^{\piso} \N$ is irreducible. The representation
$\pi:C(\U_{p})\rightarrow B(\ell^{2}(\N))$ is defined by
$$(\pi(h)\xi)(s)=\varepsilon_{u}(\beta_{s}(h))\xi(s)=\beta_{s}(h)(u)\xi(s)=h(q^{-s}u)\xi(s)$$
for every $s\in\N$, $h\in C(\U_{p})$, and $\xi\in\ell^{2}(\N)$, which is precisely the representation $\pi_{u}$, and $W=T^{*}$.
Simply speaking, the irreducible representations $\pi_{u}\rtimes T^{*}$ are derived from the (essential) ideal
$\K(\ell^{2}(\N))\otimes C(\U_{p})$ of $C(\U_{p})\rtimes_{\beta}^{\piso} \N$. Finally, the irreducible representations
$\pi_{\clu}\rtimes^{\piso} \lambda$ are indeed derived from the quotient algebra
$$C(\U_{p})\rtimes_{\beta}^{\piso} \N/(\K(\ell^{2}(\N))\otimes C(\U_{p}))\simeq C(\U_{p})\rtimes_{\beta} \Z.$$
More precisely, since $\Z$ acts on $\U_{p}$ freely, it is known that the evaluation map $\varepsilon_{u}:C(\U_{p})\rightarrow \C$ induces a
nondegenerate representation $\pi:C(\U_{p})\rightarrow B(\ell^{2}(\Z))$ defined by
$$(\pi(h)\xi)(s)=\varepsilon_{u}(\beta_{-s}(h))\xi(s)=\beta_{-s}(h)(u)\xi(s)=h(q^{s}u)\xi(s)$$
for every $s\in\Z$, $h\in C(\U_{p})$, and $\xi\in\ell^{2}(\Z)$, which is exactly the representation $\pi_{\clu}$.
It along with the left regular representation $\lambda:\Z\rightarrow B(\ell^{2}(\Z))$ forms a covariant pair of the
group dynamical system $(C(\U_{p}),\Z,\beta)$ such that the corresponding (unital) representation $\pi_{\clu}\rtimes \lambda$ of $C(\U_{p})\rtimes_{\beta} \Z$ is irreducible. Now, $\pi_{\clu}\rtimes^{\piso} \lambda=(\pi_{\clu}\rtimes \lambda)\circ \phi$, where $\phi$
denotes the natural surjection of $C(\U_{p})\rtimes_{\beta}^{\piso} \N$ onto $C(\U_{p})\rtimes_{\beta} \Z$. Thus,
$\pi_{\clu}\rtimes^{\piso} \lambda$ is irreducible (see \cite[Remark 3.12]{LZ}).
\end{proof}

\begin{lemma}
\label{res-rep}
The restrictions of the representations $$\Pi_{(n,u)}, \Pi_{(n,\widehat{u})}, \Pi_{u}\ \textrm{and}\ \Pi_{\clu}$$ of
$(C(\Z_{p})\rtimes_{\alpha}^{\piso} \N^{2},i)$ to the ideal $C_{0}(\Z_{p}\setminus\{0\})\rtimes_{\alpha}^{\piso} \N^{2}\simeq \A\otimes C(\U_{p})\rtimes_{\beta}^{\piso} \N$
are the tensor product representations
\begin{itemize}
\item[(i)] $(\phi_{n}\rtimes V)\otimes (\pi_{u}\rtimes T^{*})$,
\ \\
\item[(ii)] $(\phi_{n}\rtimes V)\otimes (\pi_{\clu}\rtimes^{\piso} \lambda)$,
\ \\
\item[(iii)] $(M\rtimes^{\piso} T)\otimes (\pi_{u}\rtimes T^{*})$ and
\ \\
\item[(iv)] $(M\rtimes^{\piso} T)\otimes (\pi_{\clu}\rtimes^{\piso} \lambda)$,

\end{itemize}
respectively. Therefore, the representations $\Pi_{(n,u)}$, $\Pi_{(n,\widehat{u})}$, $\Pi_{u}$, and $\Pi_{\clu}$ are irreducible.
\end{lemma}

\begin{proof}
We only need to compute the restrictions. We only do this for (i) as the rest follow by similar computations. Firstly, by \cite[Theorem 4.8]{SZ3},
the algebra $$C_{0}(\Z_{p}\setminus\{0\})\rtimes_{\alpha}^{\piso} \N^{2}\simeq
(\c_{0}\otimes C(\U_{p}))\rtimes_{\tau\otimes \beta}^{\piso} (\N\times \N)$$ sits in $(C(\Z_{p})\rtimes_{\alpha}^{\piso} \N^{2},i)$ as the
ideal $$\E=\clsp\{i_{\N^{2}}(r,s)^{*}i_{C(\Z_{p})}(f)i_{\N^{2}}(x,y): f\in C_{0}(\Z_{p}\setminus\{0\})\ \textrm{and}\ r,s,x,y\in\N \}.$$
Now, consider the following diagram:
\begin{equation*}
\begin{diagram}\dgARROWLENGTH=2.5\dgARROWLENGTH
\node{(\c_{0}\rtimes_{\tau}^{\textrm{piso}} \N,i)\otimes (C(\U_{p})\rtimes_{\beta}^{\piso} \N,j)}
\arrow{se,l}{\Gamma}\arrow{e,t}{(\phi_{n}\rtimes V)\otimes (\pi_{u}\rtimes T^{*})}
\node{B(\H_{n}\otimes \ell^{2}(\N))}\\
\node{} \node{\E,}\arrow{n,l}{\Pi_{(n,u)}|_{\E}}
\end{diagram}
\end{equation*}
where $\Gamma$ denotes the isomorphism of $(\c_{0}\rtimes_{\tau}^{\textrm{piso}} \N)\otimes (C(\U_{p})\rtimes_{\beta}^{\piso} \N)$ onto
$\E\simeq (\c_{0}\otimes C(\U_{p}))\rtimes_{\tau\otimes \beta}^{\piso} (\N\times \N)$ (see \cite[Theorem 5.9]{SZ3}). It is enough to see that
the above diagram commutes by calculating on the spanning elements. For every $g\in \c_{0}$, $h\in C(\U_{p})$, $m,t,x,y,s\in \N$, and
$0\leq r\leq n$, we
have
\begin{eqnarray*}
\begin{array}{l}
(\Pi_{(n,u)}|_{\E}\circ \Gamma)(i_{\N}(m)^{*}i_{\c_{0}}(g)i_{\N}(x)\otimes j_{\N}(t)^{*}j_{C(\U_{p})}(h)j_{\N}(y))(e_{r}\otimes e_{s})\\
=\Pi_{(n,u)}(i_{\N^{2}}(m,t)^{*}i_{C(\Z_{p})}(g\otimes h)i_{\N^{2}}(x,y))(e_{r}\otimes e_{s})\\
=(P_{n}T_{m}P_{n}\otimes T_{t})\rho(g\otimes h)(P_{n}T_{x}^{*}P_{n}\otimes T_{y}^{*})(e_{r}\otimes e_{s})\ (\textrm{see (i) in Lemma \ref{lem1}})\\
=(P_{n}T_{m}P_{n}\otimes T_{t})\rho(g\otimes h)(e_{r-x}\otimes e_{s-y})\\
=(P_{n}T_{m}P_{n}\otimes T_{t})((g\otimes h)(p^{n-(r-x)}q^{-(s-y)}u)(e_{r-x}\otimes e_{s-y}))\\
=(P_{n}T_{m}P_{n}\otimes T_{t})(g(n-r+x)h(q^{-(s-y)}u)(e_{r-x}\otimes e_{s-y}))\\
=g(n-r+x)h(q^{-(s-y)}u)(P_{n}T_{m}P_{n}\otimes T_{t})(e_{r-x}\otimes e_{s-y})\\
=g(n-r+x)h(q^{-(s-y)}u)(e_{r-x+m}\otimes e_{s-y+t})\\
=g(n-r+x)e_{r-x+m}\otimes h(q^{-(s-y)}u)e_{s-y+t}
\end{array}
\end{eqnarray*}
if $s-y\geq 0$, $r-x\geq 0$, and $r-x+m\leq n$, otherwise, zero.
On the other hand,
\begin{eqnarray*}
\begin{array}{l}
(\phi_{n}\rtimes V)\otimes (\pi_{u}\rtimes T^{*})(i_{\N}(m)^{*}i_{\c_{0}}(g)i_{\N}(x)\otimes j_{\N}(t)^{*}j_{C(\U_{p})}(h)j_{\N}(y))
(e_{r}\otimes e_{s})\\
=((\phi_{n}\rtimes V)(i_{\N}(m)^{*}i_{\c_{0}}(g)i_{\N}(x))\otimes (\pi_{u}\rtimes T^{*})(j_{\N}(t)^{*}j_{C(\U_{p})}(h)j_{\N}(y))
(e_{r}\otimes e_{s})\\
=((P_{n}T_{m}P_{n})\phi_{n}(g)(P_{n}T_{x}^{*}P_{n})\otimes T_{t}\pi_{u}(h)T_{y}^{*})(e_{r}\otimes e_{s})\ (\textrm{see Proposition \ref{pi1&pi2}})\\
=(P_{n}T_{m}P_{n})\phi_{n}(g)(P_{n}T_{x}^{*}P_{n})e_{r}\otimes T_{t}\pi_{u}(h)T_{y}^{*}e_{s},
\end{array}
\end{eqnarray*}
which, by similar calculations to the above, equals to
$$g(n-r+x)e_{r-x+m}\otimes h(q^{-(s-y)}u)e_{s-y+t}$$
if $s-y\geq 0$, $r-x\geq 0$, and $r-x+m\leq n$, otherwise, zero. Thus,
$$\Pi_{(n,u)}|_{\E}\circ \Gamma=(\phi_{n}\rtimes V)\otimes (\pi_{u}\rtimes T^{*}).$$
\end{proof}
Therefore, the primitive ideals derived from the ideal $C_{0}(\Z_{p}\setminus\{0\})\rtimes_{\alpha}^{\piso}\N^{2}$ are indeed the
kernels of the irreducible representations $\Pi_{(n,u)}$, $\Pi_{(n,\widehat{u})}$, $\Pi_{u}$, and $\Pi_{\clu}$, which we denote by
$\I_{(n,u)}$, $\I_{(n,\clu)}$, $\I_{u}$, and $\I_{\clu}$, respectively.

For the primitive ideals lifted from the closed set $\Prim(\T(\Z^{2}))$, the composition of the surjection $\varphi$ in (\ref{ext-seq})
with the irreducible representations of $\T(\Z^{2})$ gives the irreducible representations of $C(\Z_{p})\rtimes_{\alpha}^{\piso}\N^{2}$
derived from $\T(\Z^{2})$. Thus, we have:
\begin{lemma}
\label{lem2}
Let $T:\N\rightarrow B(\ell^{2}(\N))$ be the isometric representation as in Lemma \ref{lem1}, and $z,w\in \TT$. The irreducible representations of $(C(\Z_{p})\rtimes_{\alpha}^{\piso}\N^{2},i)$ coming from $\T(\Z^{2})$ are given by the following maps:
\begin{itemize}
\item[(a)] $\Pi_{0}:C(\Z_{p})\rtimes_{\alpha}^{\piso} \N^{2}\rightarrow B(\ell^{2}(\N)\otimes \ell^{2}(\N))$
such that
$$\Pi_{0}(i_{C(\Z_{p})}(f)i_{\N^{2}}(m,n))=f(0)(T_{m}^{*}\otimes T_{n}^{*});$$

\item[(b)] $\dot{\Pi}_{z}:C(\Z_{p})\rtimes_{\alpha}^{\piso} \N^{2}\rightarrow B(\ell^{2}(\N))$
such that
$$\dot{\Pi}_{z}(i_{C(\Z_{p})}(f)i_{\N^{2}}(m,n))=f(0)\overline{z}^{m}T_{n}^{*};$$

\item[(c)] $\ddot{\Pi}_{z}:C(\Z_{p})\rtimes_{\alpha}^{\piso} \N^{2}\rightarrow B(\ell^{2}(\N))$
such that
$$\ddot{\Pi}_{z}(i_{C(\Z_{p})}(f)i_{\N^{2}}(m,n))=f(0)\overline{z}^{n}T_{m}^{*};$$
and
\item[(d)] $\Pi_{(z,w)}:C(\Z_{p})\rtimes_{\alpha}^{\piso} \N^{2}\rightarrow \C\simeq B(\C)$
such that
$$\Pi_{(z,w)}(i_{C(\Z_{p})}(f)i_{\N^{2}}(m,n))=f(0)\overline{z}^{m}\overline{w}^{n}$$
\end{itemize}
for all $f\in C(\Z_{p})$ and $m,n\in \N$.
\end{lemma}

\begin{proof}
We skip the details of the proof since the irreducible representations of the algebra $\T(\Z^{2})$ are known. However, we should recall that
for the surjection $\varphi$ of $(C(\Z_{p})\rtimes_{\alpha}^{\piso}\N^{2},i)$ onto $\T(\Z^{2})\subset B(\ell^{2}(\N)\otimes \ell^{2}(\N))$
(see (\ref{ext-seq})) we have
$$\varphi(i_{C(\Z_{p})}(f)i_{\N^{2}}(m,n))=f(0)(T_{m}^{*}\otimes T_{n}^{*})$$
for all $f\in C(\Z_{p})$ and $m,n\in \N$.
\end{proof}
Therefore, the kernels of the irreducible representations $\Pi_{0}$, $\dot{\Pi}_{z}$, $\ddot{\Pi}_{z}$ and $\Pi_{(z,w)}$ are the primitive ideals
derived from $\T(\Z^{2})$, which we denote by $\I_{0}$, $\dot{\I}_{z}$, $\ddot{\I}_{z}$ and $\I_{(z,w)}$, respectively.

\section{The hull-kernel closures of the subsets of $\Prim(C(\Z_{p})\rtimes_{\alpha}^{\piso} \N^{2})$}
\label{sec:closures}

To compute the closures of the subsets of $\Prim\big(C(\mathbb{Z}_{p})\rtimes_{\alpha}^{\textrm{piso}} \mathbb{N}^{2}\big)$, we first need
to determine spanning elements for the primitive ideals derived from the ideal $C_{0}(\Z_{p}\setminus\{0\})\rtimes_{\alpha}^{\piso}\N^{2}$.
We will see that, in fact, each of these ideals is a finite sum of smaller ideals, for which, we identify spanning elements (or generators).

\begin{lemma}
\label{ker-CP}
Let $k$ be a positive integer and $(A,\N^{k},\alpha)$ a semigroup dynamical system with extendible endomorphisms. Let $(\pi,V)$ be an
isometric covariant representation of $(A,\N^{k},\alpha)$. Suppose that $\ker \pi$ is an extendible $\alpha$-invariant ideal of $A$, and
$U$ is a unitary representation of $\TT^{k}$ such that $(\pi\rtimes^{\iso} V,U)$ is a covariant representation of the dual system
$(A\rtimes_{\alpha}^{\iso} \N^{k}, \TT^{k}, \hat{\alpha})$. Then,
$$\ker(\pi\rtimes^{\piso} V)=(\ker \pi)\rtimes_{\alpha}^{\piso} \N^{k}+\ker (\sigma),$$
where $\sigma$ is the natural surjection of $(A\rtimes_{\alpha}^{\piso} \N^{k},i)$ onto $(A\rtimes_{\alpha}^{\iso} \N^{k},j)$
such that
$$\ker (\sigma)=\clsp\{i_{\N^{k}}(x)^{*}i_{A}(a)(1-i_{\N^{k}}(s)^{*}i_{\N^{k}}(s))i_{\N^{k}}(y): a\in A,\ x,y,s\in \N^{k}\}
\ \ (\textrm{see \cite{SZ2}}),$$
and
$$(\ker \pi)\rtimes_{\alpha}^{\piso} \N^{k}=\clsp\{i_{\N^{k}}(x)^{*}i_{A}(a)i_{\N^{k}}(y): a\in \ker \pi,\ x,y\in \N^{k}\}.$$

\end{lemma}

\begin{proof}
The inclusion $(\ker \pi)\rtimes_{\alpha}^{\piso} \N^{k}+\ker (\sigma)\subset \ker(\pi\rtimes^{\piso} V)$ is immediate. For the other inclusion,
first note that $\pi\rtimes^{\piso} V=(\pi\rtimes^{\iso} V)\circ \sigma$, where
$$\ker(\pi\rtimes^{\iso} V)=(\ker \pi)\rtimes_{\alpha}^{\iso} \N^{k}$$
by \cite[Lemma 4.2]{LPR}. Then, the restriction of the surjection $\sigma$ to the ideal
$(\ker \pi)\rtimes_{\alpha}^{\piso} \N^{k}$ gives the natural surjection of $(\ker \pi)\rtimes_{\alpha}^{\piso} \N^{k}$ onto
$(\ker \pi)\rtimes_{\alpha}^{\iso} \N^{k}$, which we denote by $\sigma_{0}$. Now, if $(\pi\rtimes^{\piso} V)(\xi)=0$, it follows that
$\sigma(\xi)\in (\ker \pi)\rtimes_{\alpha}^{\iso} \N^{k}$. Therefore, there is $\eta\in (\ker \pi)\rtimes_{\alpha}^{\piso} \N^{k}$ such
that $\sigma_{0}(\eta)=\sigma(\xi)$. However, $\sigma_{0}(\eta)=\sigma(\eta)$, and hence, $\sigma(\eta)=\sigma(\xi)$, which implies that
$(\xi-\eta)\in \ker(\sigma)$. Thus, there is $\gamma\in \ker(\sigma)$ such that $\xi-\eta=\gamma$, and therefore,
$\xi=\eta+\gamma$ which belongs to $(\ker \pi)\rtimes_{\alpha}^{\piso} \N^{k}+\ker (\sigma)$.
\end{proof}

Note that for convenience in the rest of the present work, we borrow the notation $\overline{q^{\Z}u}$ from \cite{LPR} which denotes the closure of the subset $\{q^{s}u:s\in \Z\}$ of
$\Z_{p}$, where $u\in \Z_{p}$. Moreover, $\overline{q^{\N}u}=\overline{q^{\Z}u}$ by \cite[Lemma 4.3]{LPR}. In particular, if $u\in \U_{p}$, these closures are just the equivalence class
$\clu$ of $u$ in $\U_{p}/\Gamma_{q}$.

\begin{lemma}
\label{res-ker}
Consider the irreducible representations $\pi_{u}\rtimes T^{*}$ and $\pi_{\clu}\rtimes^{\piso} \lambda$ of $(C(\U_{p})\rtimes_{\beta}^{\piso} \N,j)$, where
$u\in\U_{p}$. Let $\varepsilon_{u}$ be the evaluation map defined on $C(\U_{p})$. Then,
$$\ker\pi_{u}=\ker\pi_{\clu}=\{h\in C(\U_{p}): h\equiv 0\ \textrm{on}\ \clu\},$$ which is an extendible $\beta$-invariant
ideal of $C(\U_{p})$. Moreover,
\begin{itemize}
\item[(i)] $\ker (\pi_{u}\rtimes T^{*})=(\ker\pi_{u})\rtimes_{\beta}^{\piso} \N+\K(\ell^{2}(\N))\otimes \ker(\varepsilon_{u})$; and

\item[(ii)] $\ker (\pi_{\clu}\rtimes^{\piso} \lambda)=(\ker\pi_{\clu})\rtimes_{\beta}^{\piso} \N+\K(\ell^{2}(\N))\otimes C(\U_{p})$.
\end{itemize}
Consequently, $\ker (\pi_{u}\rtimes T^{*})\subset \ker (\pi_{\clu}\rtimes^{\piso} \lambda)$ for every $u\in\U_{p}$.
\end{lemma}

\begin{proof}
First of all, by an argument similar to the one given in \cite{LPR}, one can see that $\overline{q^{-\N}u}=\overline{\{q^{-s}u:s\in \N\}}$ equals to
$\overline{q^{\N}u}=\overline{q^{\Z}u}=\clu$. Therefore,
$$\ker\pi_{u}=\{h\in C(\U_{p}): h\equiv 0\ \textrm{on}\ \overline{q^{-\N}u}\}=\{h\in C(\U_{p}): h\equiv 0\ \textrm{on}\ \clu\}$$
and
$$\ker\pi_{\clu}=\{h\in C(\U_{p}): h\equiv 0\ \textrm{on}\ \overline{q^{\Z}u}\}=\{h\in C(\U_{p}): h\equiv 0\ \textrm{on}\ \clu\}.$$
It thus follows that $\ker\pi_{u}=\ker\pi_{\clu}$. Moreover, since $\U_{p}\backslash \clu$ is a finite union of equivalence classes $\widehat{v}$, where $v\in \U_{p}$, and each equivalence class is clearly invariant under the multiplication by the powers of
$q$, by \cite[Theorem 4.3]{Larsen}, $\ker\pi_{\clu}=\ker\pi_{u}$ is an extendible $\beta$-invariant
ideal of $C(\U_{p})$. Therefore, $(\ker\pi_{u})\rtimes_{\beta}^{\piso} \N=(\ker\pi_{\clu})\rtimes_{\beta}^{\piso} \N$ sits in $C(\U_{p})\rtimes_{\beta}^{\piso} \N$ as the ideal
$$\clsp\{j_{\N}(x)^{*}j_{C(\U_{p})}(h)j_{\N}(y): h\in \ker\pi_{\clu}\ \textrm{and}\ x,y\in\N\}.$$

Now, for (i), note that the restriction of $\pi_{u}\rtimes T^{*}$ to the ideal $\K(\ell^{2}(\N))\otimes C(\U_{p})$ of
$C(\U_{p})\rtimes_{\beta}^{\piso} \N$ is the irreducible representation $\id\otimes \varepsilon_{u}$, where
$\ker(\id\otimes \varepsilon_{u})=\K(\ell^{2}(\N))\otimes \ker(\varepsilon_{u})$. Therefore,
$$L:=(\ker\pi_{u})\rtimes_{\beta}^{\piso} \N+\K(\ell^{2}(\N))\otimes \ker(\varepsilon_{u})$$
is an ideal of $C(\U_{p})\rtimes_{\beta}^{\piso} \N$ which is clearly contained in $\ker (\pi_{u}\rtimes T^{*})$. To see the other inclusion,
let $\psi:C(\U_{p})\rtimes_{\beta}^{\piso} \N\rightarrow B(H_{\psi})$ be a nondegenerate representation such that $\ker \psi=L$. Assume that
$\xi\in \ker (\pi_{u}\rtimes T^{*})$ and $\{\eta_{i}\}$ is an approximate unit for the algebra $\K(\ell^{2}(\N))\otimes C(\U_{p})$. Since
\begin{eqnarray*}
\begin{array}{rcl}
\ker (\pi_{u}\rtimes T^{*})\cap \K(\ell^{2}(\N))\otimes C(\U_{p})&=&\ker ((\pi_{u}\rtimes T^{*})|_{\K(\ell^{2}(\N))\otimes C(\U_{p})})\\
&=&\ker(\id\otimes \varepsilon_{u})=\K(\ell^{2}(\N))\otimes \ker(\varepsilon_{u}),
\end{array}
\end{eqnarray*}
$\xi \eta_{i}\in \K(\ell^{2}(\N))\otimes \ker(\varepsilon_{u})$, and hence, $\xi \eta_{i}\in L$ for each $i$. It follows that
\begin{align}
\label{eq8}
0=\psi(\xi \eta_{i})=\psi(\xi) \psi(\eta_{i}).
\end{align}
Since $\K(\ell^{2}(\N))\otimes C(\U_{p})$ is an essential ideal of $C(\U_{p})\rtimes_{\beta}^{\piso} \N$
(see \cite{AZ,SZ}), the restriction of $\psi$ to $\K(\ell^{2}(\N))\otimes C(\U_{p})$ remains nondegenerate. Therefore, in
(\ref{eq8}), the left hand side converges strongly to $0$ while the right hand side to $\psi(\xi)$. Thus, $\psi(\xi)=0$, which implies
that $\xi\in L$, and hence, $\ker (\pi_{u}\rtimes T^{*})\subset L$.

To see (ii), consider the map $U:\TT\rightarrow B(\ell^{2}(\Z))$ defined by $U_{z}(f)(m)=z^{m}f(m)$, which is a unitary representation of $\TT$ on
$\ell^{2}(\Z)$. Now, one can compute on the spanning elements to see that the pair $(\pi_{\clu}\rtimes^{\iso} \lambda,U)$ is a covariant
representation of the dual system $(C(\U_{p})\rtimes_{\beta}^{\iso} \N, \TT, \hat{\beta})$. Therefore, by Lemma \ref{ker-CP},
$$\ker(\pi_{\clu}\rtimes^{\piso} \lambda)=(\ker\pi_{\clu})\rtimes_{\beta}^{\piso} \N+\ker (\sigma),$$
where $\sigma$ is the natural surjection of $C(\U_{p})\rtimes_{\beta}^{\piso} \N$ onto
$C(\U_{p})\rtimes_{\beta}^{\iso} \N\simeq C(\U_{p})\rtimes_{\beta} \Z$. However, see in \cite{AZ} that
$\ker (\sigma)\simeq \K(\ell^{2}(\N))\otimes C(\U_{p})$, and hence, (ii) also holds. Finally, the inclusion
$\ker (\pi_{u}\rtimes T^{*})\subset \ker (\pi_{\clu}\rtimes^{\piso} \lambda)$ is obvious.
\end{proof}

\begin{remark}
\label{res-ker-2}
The spanning elements for $\ker (\phi_{n}\rtimes V)$ can be given by \cite[Proposition 5.7(d)]{AZ2}. This is because, as it was mentioned in the
proof of the Proposition \ref{pi1&pi2}, $\phi_{n}\rtimes V$ is actually the restriction of the irreducible representation $\pi_{n}^{*}$ of $\c\rtimes_{\tau}^{\piso} \N$ studied in \cite{AZ2} to the ideal $\c_{0}\rtimes_{\tau}^{\piso} \N$.
For $\ker (M\rtimes^{\piso} T)$, as it was recalled earlier, it is a full corner of the algebra $\K(\ell^{2}(\N))\otimes \c_{0}$ of
compact operators. More precisely, there is a projection $\mu\in \L(\ell^{2}(\N)\otimes \c_{0})$ such that
$\ker (M\rtimes^{\piso} T)$ is isomorphic to the full corner $\mu(\K(\ell^{2}(\N))\otimes \c_{0})\mu$
(see the proof of the Proposition \ref{pi1&pi2} and \cite{AZ}).
\end{remark}

The following important fact will also be applied:

\begin{theorem}
\label{ess-I-cp}
Let $P$ be the positive cone of an abelian lattice-ordered (discrete) group $G$. Let $(A, P, \alpha)$ be a dynamical system consisting of a
$C^{*}$-algebra $A$ and an action $\alpha$ of $P$ on $A$ by extendible endomorphisms. If $I$ is an essential $\alpha$-invariant
extendible ideal of $A$, then $I\rtimes_{\alpha}^{\piso} P$ is an essential ideal of $A\rtimes_{\alpha}^{\piso} P$.
\end{theorem}

\begin{proof}
Since it is known by \cite[Theorem 4.8]{SZ3} that the crossed product $I\rtimes_{\alpha}^{\piso} P$ sits in $A\rtimes_{\alpha}^{\piso} P$ as
an ideal, we only show that it is an essential ideal. Let $(B, G, \gamma)$ and $(J, G, \iota)$ be the group dynamical systems obtained from
the systems $(A, P, \alpha)$ and $(I, P, \alpha)$, respectively, by the dilation process given in \cite{SZ2}. So, by \cite[Theorem 4.1]{SZ2},
the algebras $I\rtimes_{\alpha}^{\piso} P$ and $A\rtimes_{\alpha}^{\piso} P$ are full corners in the group crossed products
$J\rtimes_{\iota} G$ and $B\rtimes_{\gamma} G$, respectively. However, one can see that the algebra $J$ is an ideal of the algebra
$B\subset \ell^{\infty}(G,A)$, and the action $\iota$ is indeed the restriction of the action $\gamma$ to the ideal $J$ induced by the shift on $\ell^{\infty}(G,A)$. Therefore, the ideal $J$ is invariant under $\gamma$, and hence, $J\rtimes_{\gamma} G$ is actually an ideal of
$B\rtimes_{\gamma} G$. So, to complete the proof, we only need to show that $J\rtimes_{\gamma} G$ is an essential ideal. To do so, by
\cite[Proposition 2.4]{Kusuda}, it is enough to see that $J$ is an essential ideal of $B$. Suppose that $fJ=0$, where $f\in B$. Thus,
$$f\varphi_{s}(i)=0$$
for all $s\in G$ and $i\in I$, where $\varphi_{s}(i)\in \ell^{\infty}(G,I)$ is a spanning element of the ideal $J$ (see \cite[\S3]{SZ2}), such that
\[
\varphi_{s}(i)(t)=
   \begin{cases}
      \alpha_{ts^{-1}}(i) &\textrm{if}\empty\ \text{$s\leq t$},\\
      0 &\textrm{otherwise.}
   \end{cases}
\]
It follows that $f(s)\varphi_{s}(i)(s)=f(s)i=0$ in $A$ for every $i\in I$. Therefore, we must have $f(s)=0$ as $I$ is an essential ideal of $A$.
Since this is true for every $s\in G$, we have $f=0$, which implies that $J$ is an essential ideal of $B$. This completes the proof.
\end{proof}

\begin{cor}
\label{ess-I}
The ideal $C_{0}(\mathbb{Z}_{p}\setminus\{0\})\rtimes_{\alpha}^{\piso} \N^{2}\simeq\A\otimes (C(\U_{p})\rtimes_{\beta}^{\piso} \N)$
of $C(\mathbb{Z}_{p})\rtimes_{\alpha}^{\piso} \N^{2}$ is essential.
\end{cor}

\begin{proof}
Since $\Prim C_{0}(\mathbb{Z}_{p}\setminus\{0\})\simeq \mathbb{Z}_{p}\setminus\{0\}\simeq \N \times \U_{p}$ is dense in
$\Prim C(\mathbb{Z}_{p})\simeq \mathbb{Z}_{p}$, $C_{0}(\mathbb{Z}_{p}\setminus\{0\})$ is essential ideal of $C(\mathbb{Z}_{p})$. Thus,
$C_{0}(\mathbb{Z}_{p}\setminus\{0\})\rtimes_{\alpha}^{\piso} \N^{2}$ is an essential ideal by the Theorem \ref{ess-I-cp}.
\end{proof}

\begin{remark}
\label{rmk2}
Recall that by \cite[Lemma 3.8]{SZ4}, the crossed product
$(C(\mathbb{Z}_{p})\rtimes_{\alpha}^{\piso} \mathbb{N}^{2},i)$ contains two essential ideal $\E_{1}$ and $\E_{2}$
corresponding to the generators of the group $\Z^{2}$ such that $\E_{1}+\E_{2}$ is the kernel of the natural surjection $\Phi$ of
$C(\mathbb{Z}_{p})\rtimes_{\alpha}^{\piso} \mathbb{N}^{2}$ onto the isometric crossed product
$C(\mathbb{Z}_{p})\rtimes_{\alpha}^{\iso} \mathbb{N}^{2}$ of the system $(C(\mathbb{Z}_p),\mathbb{N}^{2},\alpha)$. We have

$$\E_{1}=\clsp\big\{\eta_{(r,s)}^{(x,y)}(f): r,s,x,y\in \N, f\in C(\mathbb{Z}_{p})\big\}$$
and
$$\E_{2}=\clsp\big\{\xi_{(r,s)}^{(x,y)}(f): r,s,x,y\in \N, f\in C(\mathbb{Z}_{p})\big\},$$
where
$$\eta_{(r,s)}^{(x,y)}(f)=i_{\N^{2}}(r,s)^{*} i_{C(\mathbb{Z}_{p})}(f) [1-i_{\N^{2}}(0,1)^{*}i_{\N^{2}}(0,1)]i_{\N^{2}}(x,y)$$
and
$$\xi_{(r,s)}^{(x,y)}(f)=i_{\N^{2}}(r,s)^{*} i_{C(\mathbb{Z}_{p})}(f) [1-i_{\N^{2}}(1,0)^{*}i_{\N^{2}}(1,0)]i_{\N^{2}}(x,y).$$
Moreover, by \cite[Theorem 3.10]{SZ4}, the essential ideal $\E_{1}\cap\E_{2}$ is isomorphic to a full corner
$Q \K(\ell^{2}(\N^{2})\otimes C(\Z_{p})) Q$ in the algebra of compact operators
$\K(\ell^{2}(\N^{2})\otimes C(\Z_{p}))\simeq \K(\ell^{2}(\N^{2}))\otimes C(\Z_{p})$,
where $Q$ is a projection in $\L(\ell^{2}(\N^{2})\otimes C(\Z_{p}))$.
Therefore, the intersection
\begin{align}
\label{ess-I2}
\E_{3}:=\A\otimes (C(\U_{p})\rtimes_{\beta}^{\piso} \N)\cap Q \K(\ell^{2}(\N^{2})\otimes C(\Z_{p})) Q
\end{align}
gives us a smaller essential ideal of $C(\Z_{p})\rtimes_{\alpha}^{\piso} \N^{2}$. Calculations on the spanning elements,
which we skip here, shows that $\E_{3}$ is, in fact, equal to the ideal
$$\mu(\K(\ell^{2}(\N))\otimes \c_{0})\mu\otimes (\K(\ell^{2}(\N))\otimes C(\U_{p})),$$
where the full corner $\mu(\K(\ell^{2}(\N))\otimes \c_{0})\mu$ and $\K(\ell^{2}(\N))\otimes C(\U_{p})$ are essential ideals in
the algebras $\A= \c_{0}\rtimes_{\tau}^{\textrm{piso}} \N$ and $C(\U_{p})\rtimes_{\beta}^{\piso} \N$, respectively
(see \cite{AZ,SZ}). Moreover,
\begin{align}
\label{ess-I3}
\E_{3} \simeq Q \K(\ell^{2}(\N^{2})\otimes C(\Z_{p}\setminus\{0\})) Q,
\end{align}
which is a full corner in
\begin{eqnarray*}
\begin{array}{rcl}
\K(\ell^{2}(\N^{2}))\otimes C(\Z_{p}\setminus\{0\})&\simeq& \K(\ell^{2}(\N))\otimes \K(\ell^{2}(\N))\otimes (\c_{0}\otimes C(\U_{p}))\\
&\simeq& (\K(\ell^{2}(\N))\otimes \c_{0})\otimes (\K(\ell^{2}(\N))\otimes C(\U_{p})).
\end{array}
\end{eqnarray*}
Note that since $\Prim \E_{3}\simeq \Prim C_{0}(\mathbb{Z}_{p}\setminus\{0\})$ is dense in
$\Prim (Q \K(\ell^{2}(\N^{2})\otimes C(\Z_{p})) Q)\simeq \Prim C(\mathbb{Z}_{p})$, and the latter one, itself, is dense in $\Prim(C(\Z_{p})\rtimes_{\alpha}^{\piso} \N^{2})$, it follows that $\Prim \E_{3}$ is dense in
$\Prim(C(\Z_{p})\rtimes_{\alpha}^{\piso} \N^{2})$. This also shows that $\E_{3}$ is an essential ideal of
$C(\Z_{p})\rtimes_{\alpha}^{\piso} \N^{2}$, which is contained in the ideal $\A\otimes (C(\U_{p})\rtimes_{\beta}^{\piso} \N)$, and therefore,
$\A\otimes (C(\U_{p})\rtimes_{\beta}^{\piso} \N)$ must be essential (see Corollary \ref{ess-I}).
\end{remark}

\begin{remark}
\label{rmk1}
Note that the kernels of the representations $\rho$ of (i) and (ii) in Lemma \ref{lem1} are given by
\begin{itemize}
\item[(a)] $\{f\in C(\Z_{p}): f\equiv 0\ \textrm{on the closure of the elements}\ p^{n-r}q^{-s}u,\ \textrm{where}\
0\leq r\leq n\ \textrm{and}\ s\in\N\}$ and
\vspace{5mm}
\item[(b)] $\{f\in C(\Z_{p}): f\equiv 0\ \textrm{on the closure of the elements}\ p^{n-r}q^{s}u,\ \textrm{where}\
0\leq r\leq n\ \textrm{and}\ s\in\Z\}$,
\end{itemize}
respectively. Both are $\alpha$-invariant ideals of $C(\Z_{p})$, however, not extendible (for example, $p(p^{n}u)=p^{n+1}u\not\in \ker \rho$).
Moreover, since $\overline{q^{-\N}u}=\overline{q^{\Z}u}=\clu$, one can see that the closures of the elements $p^{n-r}q^{-s}u$ and $p^{n-r}q^{s}u$
are the same, which is equal to $$\cup_{r=0}^{n} p^{n-r}\clu,$$
where $p^{n-r}\clu=\{p^{n-r}v: v\in \clu\}$ for every $0\leq r\leq n$. Therefore, the ideals (a) and (b) in the above, namely, the kernels of the representations $\rho$ of (i) and (ii) in Lemma \ref{lem1} are the same, which we denote by $\ker\rho_{(n,\clu)}$ (note that it can also be
denoted by $\ker \rho_{(n,u)}$ which makes sense, too). Consequently,
$$\ker\rho_{(n,\clu)}=\{f\in C(\Z_{p}): f\equiv 0\ \textrm{on}\ \cup_{r=0}^{n} p^{n-r}\clu\}.$$
Now, by calculation on the spanning elements, one can see that
$$\clsp\{i_{\N^{2}}(r,s)^{*} i_{C(\Z_{p})}(f) i_{\N^{2}}(x,y): f\in \ker\rho_{(n,\clu)}\ \textrm{and}\ r,s,x,y\in\N\}$$
is a closed ideal of $(C(\mathbb{Z}_{p})\rtimes_{\alpha}^{\piso} \mathbb{N}^{2},i)$. We denote this ideal by $\J_{(n,\clu)}$,
which is clearly contained in the ideals $\I_{(n,u)}$ and $\I_{(n,\clu)}$, namely, the kernels of the irreducible representations
$\Pi_{(n,u)}$ and $\Pi_{(n,\clu)}$, respectively (see \S3). Note that, if $\clu\in\U_{p}/\Gamma_{q}$ is fixed,
$\ker\rho_{(n+1,\clu)}\subset \ker\rho_{(n,\clu)}$, and hence, $\J_{(n+1,\clu)}\subset \J_{(n,\clu)}$ for all $n\in\N$.

In addition, for every $n\in \N$, let $L_{n}$ be the ideal of $(C(\mathbb{Z}_{p})\rtimes_{\alpha}^{\piso} \mathbb{N}^{2},i)$
generated by the elements $\{i_{\N^{2}}(x,0): x\geq n+1\}$. It is not difficult to see that $L_{n}$ is also contained in the ideals
$\I_{(n,u)}$ and $\I_{(n,\clu)}$ for every $u\in \U_{p}$, and $L_{n+1}\subset L_{n}$ for all $n\in\N$.

\end{remark}

\begin{lemma}
\label{span-el}
For every $n\in\N$ and $u\in\U_{p}$, we have
\begin{itemize}
\item[(i)] $\I_{(n,u)}=\J_{(n,\clu)}+\ker(\phi_{n}\rtimes V)\otimes (C(\U_{p})\rtimes_{\beta}^{\piso} \N)+\A\otimes \ker(\pi_{u}\rtimes T^{*})+L_{n}$,
\vspace{5mm}
\item[(ii)] $\I_{(n,\clu)}=\J_{(n,\clu)}+\ker(\phi_{n}\rtimes V)\otimes (C(\U_{p})\rtimes_{\beta}^{\piso}\N)+\A\otimes\ker(\pi_{\clu}\rtimes^{\piso} \lambda)+\E_{1}+L_{n}$,
\vspace{5mm}
\item[(iii)] $\I_{u}=\E_{2}+\A\otimes \ker(\pi_{u}\rtimes T^{*})$, and
\vspace{5mm}
\item[(iv)] $\I_{\clu}=\ker \Phi+\A\otimes \ker(\pi_{\clu}\rtimes^{\piso} \lambda)$,
\end{itemize}
where $\Phi$ is the natural surjection of $C(\mathbb{Z}_{p})\rtimes_{\alpha}^{\piso} \mathbb{N}^{2}$ onto $C(\mathbb{Z}_{p})\rtimes_{\alpha}^{\iso} \mathbb{N}^{2}$ such that $\ker \Phi=\E_{1}+\E_{2}$ (see Remark \ref{rmk2}). Consequently, $\I_{(n,u)}\subset \I_{(n,\clu)}$ and
$\I_{u}\subset \I_{\clu}$ for every $n\in\N$ and $u\in\U_{p}$.
\end{lemma}

\begin{proof}
For (i), let
$$\E:=\J_{(n,\clu)}+\ker(\phi_{n}\rtimes V)\otimes (C(\U_{p})\rtimes_{\beta} \N)+\A\otimes \ker(\pi_{u}\rtimes T^{*})+L_{n},$$
which is obviously an ideal of $C(\mathbb{Z}_{p})\rtimes_{\alpha}^{\piso} \mathbb{N}^{2}$.
We only show that $\I_{(n,u)}\subset \E$ as the other inclusion follows immediately.
Let $\psi:C(\mathbb{Z}_{p})\rtimes_{\alpha}^{\piso} \mathbb{N}^{2}\rightarrow B(H_{\psi})$ be a nondegenerate representation
such that $\ker \psi=\E$. Let $\xi\in \I_{(n,u)}=\ker \Pi_{(n,u)}$. Take an approximate unit $\{\xi_{\lambda}\}$ in the ideal
$\A\otimes (C(\U_{p})\rtimes_{\beta} \N)$. We have $\xi \xi_{\lambda}\in \ker \Pi_{(n,u)}\cap (\A\otimes (C(\U_{p})\rtimes_{\beta} \N))$
for each $\lambda$, where
$$\ker \Pi_{(n,u)}\cap (\A\otimes (C(\U_{p})\rtimes_{\beta} \N))=\ker (\Pi_{(n,u)}|_{\A\otimes (C(\U_{p})\rtimes_{\beta} \N)}).$$
Since
$$\Pi_{(n,u)}|_{\A\otimes (C(\U_{p})\rtimes_{\beta} \N)}=(\phi_{n}\rtimes V)\otimes (\pi_{u}\rtimes T^{*})\ (\textrm{see Lemma}\ \ref{res-rep}),$$
it follows that $\xi \xi_{\lambda}$ belongs to
$$\ker(\phi_{n}\rtimes V)\otimes (C(\U_{p})\rtimes_{\beta} \N)+\A\otimes \ker(\pi_{u}\rtimes T^{*}),$$
which is contained in $\E=\ker \psi$. Therefore,
\begin{align}
\label{eq7}
\psi(\xi \xi_{\lambda})=\psi(\xi)\psi(\xi_{\lambda})=0
\end{align}
for each $\lambda$. Now, the restriction of the presentation $\psi$ to $\A\otimes (C(\U_{p})\rtimes_{\beta}^{\piso} \N)$ remains
nondegenerate as $\A\otimes (C(\U_{p})\rtimes_{\beta}^{\piso} \N)$ is an essential (see Corollary \ref{ess-I}). So, the left hand
side of (\ref{eq7}) converges strongly to $\psi(\xi)$ while the right hand side to zero. Thus, we must have $\psi(\xi)=0$ which
implies that $\ker \Pi_{(n,u)}\subset \ker \psi=\E$.

We skip the proofs of (ii)-(iv) as we can act similar to (i). However, it should be mentioned that the ideals $\E_{1}$ and $\E_{2}$ contain
\begin{align}\label{I1}
\A\otimes (\K(\ell^{2}(\N))\otimes C(\U_{p}))
\end{align}
and
\begin{align}\label{I2}
\mu(\K(\ell^{2}(\N))\otimes \c_{0})\mu\otimes C(\U_{p})\rtimes_{\beta} \N=\ker (M\rtimes^{\piso} T)\otimes C(\U_{p})\rtimes_{\beta} \N,
\end{align}
respectively. This is because calculations, which are skipped here, show that each spanning element of
the ideal (\ref{I1}) and the ideal (\ref{I2}) equals to $\eta_{(r,s)}^{(x,y)}(f)\in \E_{1}$ and $\xi_{(r,s)}^{(x,y)}(f)\in \E_{2}$, where
$f=(g\otimes h)\in \c_{0}\otimes C(\U_{p})\simeq C_{0}(\mathbb{Z}_{p}\setminus\{0\})$, respectively.

At last, the inclusions $\I_{(n,u)}\subset \I_{(n,\clu)}$ and $\I_{u}\subset \I_{\clu}$ follow immediately as
$\ker (\pi_{u}\rtimes T^{*})\subset \ker (\pi_{\clu}\rtimes^{\piso} \lambda)$ (see Lemma \ref{res-ker}).
\end{proof}

\begin{theorem}\label{main-TH}
Let $n\in\N$ and $u\in\U_{p}$. The maps
\begin{itemize}
\item[$\bullet$] $(n,u)\mapsto \I_{(n,u)}$,
\item[$\bullet$] $(n,\clu)\mapsto \I_{(n,\clu)}$,
\item[$\bullet$] $u\mapsto \I_{u}$,
\item[$\bullet$] $\clu\mapsto \I_{\clu}$,
\item[$\bullet$] $0\mapsto \I_{0}$,
\item[$\bullet$] $z\mapsto \dot{\I}_{z}$,
\item[$\bullet$] $z\mapsto \ddot{\I}_{z}$, and
\item[$\bullet$] $(z,w)\mapsto \I_{(z,w)}$
\end{itemize}
combine to give a bijection of the disjoint union
\begin{align}
\label{dijU}
(\N\times \U_{p}) \sqcup (\N\times (\U_{p}/\Gamma_{q})) \sqcup \U_{p}\sqcup (\U_{p}/\Gamma_{q})
\sqcup\{0\}\sqcup \dot{\TT}\sqcup\ddot{\TT} \sqcup \TT^{2}
\end{align}

onto $\Prim (C(\Z_{p})\rtimes_{\alpha}^{\piso} \N^{2})$.
For the hull-kernel closure $\overline{F}$ of a subset $F$ of (\ref{dijU}), we have:
\begin{itemize}
\item[(a)] if $F$ is any subset of $\N\times \U_{p}$ consisting of pairs $(n,u)$ with finitely many distinct elements $n$
from $\N$, then $\overline{F}$ is the usual closure of $F$ in
$$\Prim(\A\otimes (C(\U_{p})\rtimes_{\beta}^{\piso} \N))\simeq (\N\cup\{\infty\})\times (\U_{p}\sqcup \U_{p}/\Gamma_{q});$$
\item[(b)] if $F$ is any subset of $\N\times \U_{p}$ consisting of pairs $(n,u)$ with infinitely many distinct elements $n$
from $\N$, then $\overline{F}$ is the usual closure of $F$ in
$\Prim(\A\otimes (C(\U_{p})\rtimes_{\beta}^{\piso} \N))$ together with $\{0\}\sqcup \dot{\TT}\sqcup\ddot{\TT} \sqcup \TT^{2}$;
\medskip
\item[(c)] if $F$ is any subset of $\N\times (\U_{p}/\Gamma_{q})$ consisting of pairs $(n,\clu)$ with finitely many distinct elements $n$
from $\N$, then $\overline{F}=F$;
\medskip
\item[(d)] if $F$ is any subset of $\N\times (\U_{p}/\Gamma_{q})$ consisting of pairs $(n,\clu)$ with infinitely many distinct elements $n$
from $\N$, then $\overline{F}$ is the usual closure of $F$ in
$\Prim(\A\otimes (C(\U_{p})\rtimes_{\beta}^{\piso} \N))$ together with $\ddot{\TT} \sqcup \TT^{2}$;
\medskip
\item[(e)] if $F\subset \U_{p}$, then
$\overline{F}$ is the usual closure of $F$ in $(\U_{p}\sqcup \U_{p}/\Gamma_{q})$ together with $\dot{\TT} \sqcup \TT^{2}$;
\medskip
\item[(f)] if $F\subset \U_{p}/\Gamma_{q}$, then
$\overline{F}=F \sqcup \TT^{2}$; and
\medskip
\item[(g)] if $F\subset(\{0\}\sqcup \dot{\TT}\sqcup\ddot{\TT} \sqcup \TT^{2})$, then $\overline{F}$ is the usual closure of $F$ in
$\{0\}\sqcup \dot{\TT}\sqcup\ddot{\TT} \sqcup \TT^{2}$.
\end{itemize}
\end{theorem}

\begin{proof}
To see (a), first note that $\bigcap_{(n,u)\in F}L_{n}\subset \bigcap_{(n,u)\in F}\I_{(n,u)}$. Moreover, since $L_{n+1}\subset L_{n}$ for
every $n\in \N$ (see Remark \ref{rmk1}), $\bigcap_{(n,u)\in F}L_{n}=L_{m}$, where $m=\max\{n:(n,u)\in F\}$. Now, since the element (generator) $i_{\N^{2}}(m+1,0)\in L_{m}$ does not belong to any ideal
$$\I_{0}, \dot{\I}_{z}, \ddot{\I}_{z},\ \textrm{and}\ \I_{(z,w)}$$
for all $z,w\in \TT$, (a) is valid.

For (b), first of all, since $\Prim(\A\otimes (C(\U_{p})\rtimes_{\beta} \N))$ is homeomorphic to an open (dense) subset of
$\Prim(C(\Z_{p})\rtimes_{\alpha}^{\piso} \N^{2})$, $\overline{F}$ contains the closure of $F$ in
$\Prim(\A\otimes (C(\U_{p})\rtimes_{\beta} \N))$. To find other elements of
$\overline{F}$, we just show that
\begin{equation}
\label{eq10}
\bigcap_{(n,u)\in F}\I_{(n,u)}\subset \A\otimes (C(\U_{p})\rtimes_{\beta} \N),
\end{equation}
which completes the proof of (b). To do so, we extract an infinite (countable) subset $\{n_{k}: k\in\N\}$ of $\N$ from $F$. Then, for
each $n_{k}$, we pick an element $u_{k}\in\U_{p}$ such that $(n_{k},u_{k})$ belongs to $F$. Thus, we obtain an infinite (countable)
subset $E=\{(n_{k},u_{k}): k\in\N\}$ of $F$. Now, we show that
$$\bigcap_{(n_{k},u_{k})\in E}\I_{(n_{k},u_{k})}=\bigcap_{k\in \N}\I_{(n_{k},u_{k})}\subset \A\otimes (C(\U_{p})\rtimes_{\beta} \N),$$
which implies that (\ref{eq10}) is true. So, we apply \cite[Theorem 3.10]{SZ4} to get the composition series of ideals of $C(\Z_{p})\rtimes_{\alpha}^{\piso} \N^{2}$, from which, it follows that each primitive ideal $\I_{(n,u)}$ ($(n,u)\in \N\times \U_{p}$)
is indeed lifted from the ideal $\E_{1}\cap\E_{2}\simeq Q \K(\ell^{2}(\N^{2})\otimes C(\Z_{p})) Q$, and eventually, from the
ideal $\E_{3} \simeq Q \K(\ell^{2}(\N^{2})\otimes C(\Z_{p}\setminus\{0\})) Q$ of $C(\Z_{p})\rtimes_{\alpha}^{\piso} \N^{2}$ (see Remark \ref{rmk2}), where $\Prim (\E_{1}\cap\E_{2})\simeq \Z_{p}$ and $\Prim \E_{3}\simeq \Z_{p}\setminus\{0\}\simeq \N\times \U_{p}$.
Therefore, the ideals $\{\I_{(n_{k},u_{k})}: k\in \N\}$ correspond to the sequence $\{p^{n_{k}}u_{k}:k\in \N\}$ in $\Z_{p}\setminus\{0\}$, which converges to $0\in \Z_{p}$. This implies that $\cap_{k\in \N}\I_{(n_{k},u_{k})}$ must be contained in the primitive ideal of $C(\Z_{p})\rtimes_{\alpha}^{\piso} \N^{2}$ lifted from $\Prim (\E_{1}\cap\E_{2})\simeq \Z_{p}$ corresponds to $0$, which is actually
$\I_{0}=\ker \Pi_{0}=\ker \varphi=\A\otimes (C(\U_{p})\rtimes_{\beta} \N)$ (see (\ref{ext-seq}) and Lemma \ref{lem2}).
This is because $0\in \Z_{p}$ corresponds to the evaluation map $\varepsilon_{0}$ of $C(\Z_{p})$ which induces
the unital representation $\pi:C(\Z_{p})\rightarrow B(\ell^{2}(\N)\otimes \ell^{2}(\N))$
defined by $(\pi(f)\xi)(r,s)=\varepsilon_{0}(f)\xi(r,s)=f(0)\xi(r,s)$. Together with the representation
$W:\N^{2}\rightarrow B(\ell^{2}(\N)\otimes \ell^{2}(\N))$ defined by
$W_{(r,s)}=T_{r}^{*}\otimes T_{s}^{*}$, $\pi$ forms a Nica-covariant partial-isometric representation of the system
$(C(\Z_{p}),\N^{2},\alpha)$ such that the corresponding (unital) representation $\pi\rtimes W$ of $C(\Z_{p})\rtimes_{\alpha}^{\piso} \N^{2}$
is precisely the irreducible representation $\Pi_{0}$.

For (c), an argument similar to (a) shows that $\overline{F}$ is the usual closure of $F$ in
$\Prim(\A\otimes (C(\U_{p})\rtimes_{\beta}^{\piso} \N))$. However, in this case, $F$ must be a finite set as $\U_{p}/\Gamma_{q}$ is
(as a closed subset of $\Prim(C(\U_{p})\rtimes_{\beta}^{\piso} \N)$). Therefore, $F=\bigcup_{(n,\clu)\in F}(\{n\}\times \{\clu\})$, from which, it
follows that
$$\overline{F}=\cup_{(n,\clu)\in F}\overline{\{n\}\times \{\clu\}}=\cup_{(n,\clu)\in F}(\overline{\{n\}}\times \overline{\{\clu\}})
=\cup_{(n,\clu)\in F}(\{n\}\times \{\clu\})=F.$$

For (d), first, obviously, $\E_{1}\subset \bigcap_{(n,\clu)\in F}\I_{(n,\clu)}$, and therefore, $\bigcap_{(n,\clu)\in F}\I_{(n,\clu)}$
is not contained in the primitive ideals $\I_{0}$ and $\dot{\I}_{z}$
since they do not contain $\E_{1}$. However, we claim that $\overline{F}$ contains $\ddot{\TT} \sqcup \TT^{2}$. To see this, note that since
$\U_{p}/\Gamma_{q}$ is a finite set, $F$ contains an infinite (countable) subset $\{(n_{k},\clu): k\in\N\}$, where $\clu$ is fixed.
So, it is enough to see that $\ddot{\TT} \sqcup \TT^{2}$ is contained in the closure of any set $E$ of the form $A\times\{\clu\}$, where
$A$ is any infinite subset of $\N$ and $u$ is any element of $\U_{p}$. For $\TT^{2}$, since the usual closure of $E$ in
$\Prim(\A\otimes (C(\U_{p})\rtimes_{\beta} \N))$, which is $(A\cup\{\infty\})\times \{\clu\}$,
is contained in $\overline{E}$, $\bigcap_{(n,\clu)\in E}\I_{(n,\clu)}$ is contained in the primitive ideal $\I_{\clu}$ (corresponding to the element $(\infty,\clu)$). Therefore, since $\I_{\clu}\subset\I_{(z,w)}$ for all $u\in\U_{p}$ and $z,w\in\TT$, $\bigcap_{(n,\clu)\in E}\I_{(n,\clu)}$
is contained in each ideal $\I_{(z,w)}$. It follows that $\TT^{2}\subset\overline{E}$, and hence, $\TT^{2}\subset\overline{F}$.
Finally, we show that
\begin{equation*}
\bigcap_{(n,\clu)\in E}\I_{(n,\clu)}\subset \ddot{\I}_{z}
\end{equation*}
for all $z\in \TT$. Again, we apply the composition series of ideals of $C(\Z_{p})\rtimes_{\alpha}^{\piso} \N^{2}$ given in \cite[Theorem 3.10]{SZ4},
by which, it follows that the primitive ideals
$\I_{(m,\widehat{v})}$ ($(m,\widehat{v})\in \N\times \U_{p}/\Gamma_{q}$) are lifted from the ideal $\ker \Phi=\E_{1}+\E_{2}$ of
$C(\Z_{p})\rtimes_{\alpha}^{\piso} \N^{2}$ (see Remark \ref{rmk2} again).
More precisely, they are derived from the quotient algebra
$$\ker \Phi/(\E_{1}\cap\E_{2})\simeq \E_{1}/(\E_{1}\cap\E_{2})\oplus \E_{2}/(\E_{1}\cap\E_{2}),$$
and ultimately, from $\E_{2}/(\E_{1}\cap\E_{2})$,
which is Morita-equivalent to the isometric crossed product $C(\Z_{p})\rtimes_{\ddot{\alpha}}^{\iso} \N$, where
$\ddot{\alpha}_{s}=\alpha_{(0,s)}$ for all $s\in\N$. To see this, since the action $\ddot{\alpha}$ is actually given by automorphisms
(multiplications by the powers of $q$ on $\Z_{p}$), $C(\Z_{p})\rtimes_{\ddot{\alpha}}^{\iso} \N\simeq C(\Z_{p})\rtimes_{\ddot{\alpha}} \Z$.
Therefore, by \cite[Theorem 8.39]{W}, $\Prim (C(\Z_{p})\rtimes_{\ddot{\alpha}} \Z)$ is homeomorphic to a quotient of the product space
$\Z_{p}\times \TT$ equipped with the quotient topology. It follows that, while $(p^{m}v,z)\nsim (0,w)$ for all
$(m,v)\in \N\times \U_{p}\simeq \Z_{p}\setminus\{0\}$ and $z,w\in\TT$, since $\Z$ acts on $\Z_{p}\setminus\{0\}$ freely,
$(p^{m}v,z)\sim (p^{r}t,w)$ if and only if $m=r$ and $\widehat{v}=\widehat{t}$ in $\U_{p}/\Gamma_{q}$. Thus, the equivalent class of
each pair $(p^{m}v,z)$ in $\Prim (C(\Z_{p})\rtimes_{\ddot{\alpha}} \Z)$ can be parameterized by the pair $(m,\widehat{v})$ in the product space
$\N\times (\U_{p}/\Gamma_{q})$. Moreover, $(0,z)\sim (0,w)$ if and only if $z=w$, and hence, the equivalent class of
each pair $(0,z)$ can be simply parameterized by $z$ in $\TT$. Consequently, $\Prim (C(\Z_{p})\rtimes_{\ddot{\alpha}} \Z)$ can be identified
with disjoint union
\begin{equation}
\label{eq1}
(\N\times (\U_{p}/\Gamma_{q})) \sqcup \TT
\end{equation}
equipped with the quotient topology. Therefore, (\ref{eq1}) in $\Prim (C(\Z_{p})\rtimes_{\alpha}^{\piso} \N^{2})$ (see (\ref{dijU}))
corresponds to
$$\Prim (C(\Z_{p})\rtimes_{\ddot{\alpha}} \Z)\simeq \Prim(\E_{2}/(\E_{1}\cap\E_{2}))\simeq \Prim(\ker \Phi/\E_{1})$$
which is a closed subset of $\Prim (\ker \Phi)$ (as well as $\Prim(\E_{1}/(\E_{1}\cap\E_{2})$)). This implies that all primitive ideals
lifted from $\Prim (C(\Z_{p})\rtimes_{\ddot{\alpha}} \Z)$ contain the ideal $\E_{1}$, which are the ideals $\I_{(m,\widehat{v})}$ and $\ddot{\I}_{z}$.
Thus, $\TT$ in (\ref{eq1}) corresponds to $\ddot{\TT}$ in $\Prim (C(\Z_{p})\rtimes_{\alpha}^{\piso} \N^{2})$. Now, let
$$\sigma:\Z_{p}\times \TT\rightarrow \Prim (C(\Z_{p})\rtimes_{\ddot{\alpha}} \Z)$$
be the quotient map. The primitive ideals $\{\I_{(n,\clu)}:(n,\clu)\in E\}$ correspond to the sequence
$\{(n,\clu):n\in A\}$ in $\Prim (C(\Z_{p})\rtimes_{\ddot{\alpha}} \Z)$ whose closure, as an infinite set, contains $\TT\simeq \ddot{\TT}$.
This is because $(n,\clu)=\sigma(p^{n}u,z)$ for every $z\in\TT$, and since $A$ is infinite, the sequence $\{(p^{n}u,z):n\in A\}$
converges to $(0,z)$ in the product space $\Z_{p}\times \TT$. Therefore, the sequence $\{(n,\clu):n\in A\}$ converges to $\sigma(0,z)=z$
in $\Prim (C(\Z_{p})\rtimes_{\ddot{\alpha}} \Z)$ as $\sigma$ is continuous. This means that $\{(n,\clu):n\in A\}$ converges to every
$z\in\TT$ (which is not unusual as the primitive ideal spaces of $C^{*}$-algebras are non-Hausdorff in general), and hence, every $z\in\TT$
belongs to the closure of $\{(n,\clu):n\in A\}$. It thus follows that $\bigcap_{(n,\clu)\in E}\I_{(n,\clu)}\subset \ddot{\I}_{z}$ for every
$z\in \TT$. Consequently, $\overline{F}$ is the usual closure of $F$ in
$\Prim(\A\otimes (C(\U_{p})\rtimes_{\beta} \N))$ together with $\ddot{\TT} \sqcup \TT^{2}$.

For (e), since $\E_{2}\subset \I_{u}\subset \A\otimes (C(\U_{p})\rtimes^{\piso}_{\beta} \N)+\E_{2}$ for every
$u\in \U_{p}$, we have
$$\E_{2}\subset \bigcap_{u\in F}\I_{u}\subset \A\otimes (C(\U_{p})\rtimes^{\piso}_{\beta} \N)+\E_{2}.$$
Therefore, since the primitive ideals $\dot{\I}_{z}$ and $\I_{(z,w)}$ contain $\A\otimes (C(\U_{p})\rtimes_{\beta} \N)+\E_{2}$ and $\E_{2}$
is not contained in the primitive ideals $\I_{0}$ and $\ddot{\I}_{z}$, $\overline{F}$ is the usual closure of $F$ in
$\Prim(\A\otimes (C(\U_{p})\rtimes_{\beta} \N))$ together with $\dot{\TT} \sqcup \TT^{2}$. However, $F$ actually corresponds to the set
$\{\infty\}\times F$ in $\Prim(\A\otimes (C(\U_{p})\rtimes_{\beta} \N))$, whose usual closure in
$\Prim(\A\otimes (C(\U_{p})\rtimes_{\beta} \N))$ simply corresponds to the usual closure of $F$ in $\U_{p}\sqcup \U_{p}/\Gamma_{q}$.

For (f), since $\ker\Phi\subset \I_{\clu}\subset \A\otimes (C(\U_{p})\rtimes^{\piso}_{\beta} \N)+\ker\Phi$ for every $u\in \U_{p}$,
$$\ker\Phi\subset \bigcap_{\clu\in F}\I_{\clu}\subset \A\otimes (C(\U_{p})\rtimes^{\piso}_{\beta} \N)+\ker\Phi.$$
Thus, since $\A\otimes (C(\U_{p})\rtimes^{\piso}_{\beta} \N)+\ker\Phi$ is contained in each ideal $\I_{(z,w)}$ and the ideals
$\I_{0}$, $\dot{\I}_{z}$, and $\ddot{\I}_{z}$ do not contain $\ker\Phi$, $\overline{F}$ is the usual closure of $F$ in
$\Prim(\A\otimes (C(\U_{p})\rtimes_{\beta} \N))$ together with $\TT^{2}$. Again, note that the usual
closure of $F$ in $\Prim(\A\otimes (C(\U_{p})\rtimes_{\beta} \N))$ corresponds to the usual closure of $F$ in
$\U_{p}\sqcup \U_{p}/\Gamma_{q}$, which is just $F$.

Finally, (g) follows immediately as $\Prim \T(\Z^{2})=\{0\}\sqcup \dot{\TT}\sqcup\ddot{\TT} \sqcup \TT^{2}$ is (homeomorphic to) a closed subset of
$\Prim(C(\Z_{p})\rtimes_{\alpha}^{\piso} \N^{2})$.

\end{proof}

\begin{remark}
\label{rmk3}
Note that if the subset $F$ of $\Prim(C(\Z_{p})\rtimes_{\alpha}^{\piso} \N^{2})$ is $\N\times \U_{p}$, then the statement (b) of
the Theorem \ref{main-TH} implies that
$$\overline{F}=(\N\cup\{\infty\})\times (\U_{p}\sqcup \U_{p}/\Gamma_{q})\sqcup\{0\}\sqcup \dot{\TT}\sqcup\ddot{\TT} \sqcup \TT^{2},$$
that is, $F$ is dense in $\Prim(C(\Z_{p})\rtimes_{\alpha}^{\piso} \N^{2})$ (recall that $\U_{p}$ is dense in $\Prim(C(\U_{p})\rtimes_{\beta} \N)$
as $\K(\ell^{2}(\N))\otimes C(\U_{p})$ is an essential ideal of $C(\U_{p})\rtimes_{\beta} \N$). This makes sense as $F$ actually corresponds to
$\Prim \E_{3}$, where $\E_{3}$ is an essential ideal of $C(\Z_{p})\rtimes_{\alpha}^{\piso} \N^{2}$ (see Remark \ref{rmk2}).
It follows that
$\bigcap_{(n,u)\in (\N\times \U_{p})}\I_{(n,u)}=\{0\}$,
and therefore,
\begin{equation*}
\bigcap_{n\in \N}L_{n}=\{0\} \
\end{equation*}
as $\bigcap_{n\in \N}L_{n}\subset \bigcap_{(n,u)\in (\N\times \U_{p})}\I_{(n,u)}$.

Also, for the closure $\overline{F}$ of a subset $F$ of $\U_{p}\sqcup \U_{p}/\Gamma_{q}\simeq \Prim(C(\U_{p})\rtimes_{\beta} \N)$, we have
\begin{itemize}
\item[$\bullet$] $\overline{F}=F$ if $F\subset \U_{p}/\Gamma_{q}$ as the finite set $\U_{p}/\Gamma_{q}$ is a closed subset of
$\Prim(C(\U_{p})\rtimes_{\beta} \N)$ equipped with the discrete topology; and
\item[$\bullet$] $\overline{F}=\overline{F}^{\U_{p}}\cup \sigma(\overline{F}^{\U_{p}})$ if $F\subset \U_{p}$, where $\overline{F}^{\U_{p}}$
denotes the usual closure of $F$ in $\U_{p}$ and $\sigma$ the quotient map of $\U_{p}$ onto $\U_{p}/\Gamma_{q}$. This is because
$\ker (\pi_{u}\rtimes T^{*})\subset \ker (\pi_{\clu}\rtimes^{\piso} \lambda)$ for every $u\in\U_{p}$ (see Lemma \ref{res-ker}).

\end{itemize}

\end{remark}


\end{document}